\documentclass[12pt]{amsart}

\usepackage[colorlinks=true,linkcolor=blue]{hyperref}
\usepackage{amsmath}
\usepackage{amssymb}
\usepackage{amsthm}
\usepackage[shortlabels]{enumitem}

\usepackage{graphicx}
\usepackage{epstopdf}
\usepackage{epsfig}

\usepackage{enumitem}

\usepackage{hyperref}

\usepackage{times}
\usepackage{relsize}
\usepackage{textcomp}
\usepackage{amssymb}
\usepackage[english]{babel}
\usepackage[autostyle]{csquotes}
\usepackage{epstopdf}
\usepackage{nicefrac}
\usepackage{tikz}

\theoremstyle{plain} 
\newtheorem{thm}{Theorem}[section]
\newtheorem{prop}[thm]{Proposition}
\newtheorem{lemma}[thm]{Lemma}
\newtheorem{cor}[thm]{Corollary} 

\newtheorem{conj}[thm]{Conjecture}
\theoremstyle{remark}
\newtheorem{remark}[thm]{Remark}

\theoremstyle{definition}
\newtheorem{defin}[thm]{Definition}

\newcommand{\CC}{\mathbb{C}}

\newcommand{\QQ}{\mathbb{Q}}

\newcommand{\ZZ}{\mathbb{Z}}

\DeclareMathOperator{\sgn}{sgn}

\DeclareMathOperator{\Gal}{Gal}

\DeclareMathOperator{\Aut}{Aut}

\newcommand{\la}{\langle}
\newcommand{\ra}{\rangle}

\begin{document}
	
	\title{Iterates of post-critically finite polynomials of the form $\boldsymbol{x^d+c}$}
	\date{}
	\author[Goksel]{Vefa Goksel}
	\address{Towson University \\ Towson, MD 21252 \\ USA}
	\email{vgoksel@towson.edu}
	
	\begin{abstract}
		Fix a prime number $d$. The post-critically finite polynomials of the form $f_{d,c} = x^d+c\in \CC[x]$ play a fundamental role in polynomial dynamics. While many results are known in the complex dynamical setting, much less is understood about the arithmetic properties of these polynomials. In this paper, we describe the factorization of the iterates of post-critically finite polynomials $f_{d,c}$ over their fields of definition. As a consequence, we prove new cases of a conjecture of Andrews and Petsche on abelian arboreal Galois representations.
	\end{abstract}
	
	\maketitle

	\section{Introduction}
	Let $f \in \mathbb{C}[x]$ be a polynomial of degree at least $2$. We denote by $f^n$ the \emph{$n$-th iterate of $f$}, defined inductively by $f^0 = x$ and $f^n = f \circ f^{n-1}$ for $n \geq 1$.

	
	We consider the family of polynomials $f_{d,c} = x^d + c \in \mathbb{C}[x]$, where $d$ is any prime. Each polynomial in this family has $0$ as its unique (finite) critical point. Therefore, the \emph{post-critical orbit} of $f_{d,c}$ is given by the forward orbit
	\[
	\{c,\, c^d + c,\, (c^d + c)^d + c,\, \dots\}
	\]
	of $0$ under iteration by $f_{d,c}$. For simplicity, we define $a_1 = c$ and $a_{i+1} = a_i^d + c$ for $i \geq 1$. The parameters $c_0$ for which $f_{d,c_0}^n(0) = 0$ and $f_{d,c_0}^i(0) \neq 0$ for all $0 < i < n$ are called \emph{Gleason parameters of period $n$}, and they are the roots of the monic polynomial $G_{d,0,n} \in \mathbb{Z}[c]$ defined by
	\begin{equation}
		\label{eq:Gleasondef}
		G_{d,0,n}(c) := \prod_{k \mid n} (a_k)^{\mu(n/k)} \in \mathbb{Z}[c].
	\end{equation}
	
	Suppose $c_0\in \CC$ is such that $0$ is strictly preperiodic for $f_{d,c_0}$. Then $f_{d,c_0}^{m+n}(0) = f_{d,c_0}^m(0)$ for some minimal integers $m\geq 2, n\geq 1$. It follows that $\zeta f_{d,c_0}^{m-1}(0) = f_{d,c_0}^{m+n-1}(0)$ for some $d$-th root of unity $\zeta \ne 1$. These parameters are called the \emph{Misiurewicz parameters of type $(m,n)$}, and they are the roots of the monic polynomial $G_{d,m,n}^{\zeta}(c) \in \ZZ[\zeta][c]$ defined by
	\begin{equation}
		\label{eq:Gdef}
		G_{d,m,n}^{\zeta}(c) := \prod_{k \mid n} \big( a_{m+k-1} - \zeta a_{m-1} \big)^{\mu(n/k)}
		\cdot \begin{cases}
			\prod_{k \mid n} \big( a_k \big)^{-\mu(n/k)} & \text{if } n \mid (m-1), \\
			1 & \text{if } n \nmid (m-1).
		\end{cases}
	\end{equation}
	
	Misiurewicz and Gleason parameters have been extensively studied in complex dynamics, especially in the quadratic case; see \cite{DH,Eber99,FG15,HubSch94,Poirier93,Milnor20,Romera1996,Schleicher94} for a limited list of examples. In contrast, there is very little known about the arithmetic properties of these parameters. For instance, recently, the irreducibility of the polynomials $G_{d,m,n}^{\zeta}$ over $\QQ(\zeta)$ has been established in some special cases \cite{BEK19,Gok19,Gok20}. It has been conjectured that these polynomials are always irreducible over $\QQ(\zeta)$; see, for example, \cite[Remark~3.5]{Milnor12}, \cite[Exercise~4.17]{Sil07}, and \cite[Conjecture 1.1]{BG1}. See also \cite{BG1,BG2,BG3} for partial results on a related arithmetic question concerning the number fields generated by Misiurewicz parameters.
	
	 \par
	
	A polynomial $f \in \CC[x]$ is called \emph{post-critically finite (PCF)} if the orbits of all its critical points under $f$ are finite. In particular, $f_{d,c_0}$ is PCF when $c_0$ is a Misiurewicz or Gleason parameter. Moreover, in the case $d = 2$, any PCF quadratic polynomial is linearly conjugate to $f_{2,c_0}$ for a unique Misiurewicz or Gleason parameter $c_0$.\par
	In this paper, we study the factorization of the iterated polynomials $f_{2,c_0}^n - \alpha$ over $\mathbb{Q}(c_0)$, where $c_0$ is a Misiurewicz or Gleason parameter and $\alpha \in \mathbb{Q}(c_0)$. We give an explicit description of the factorization of these polynomials under a local condition on $\alpha$ (see Theorem~\ref{thm:preperiodic} and Theorem~\ref{thm:periodic}).
	
		\begin{defin}
		\label{def:stability}
		Let $K$ be a field, let $\alpha \in K$, and let $f \in K[x]$ be a polynomial of degree at least $2$.
		We say that the pair $(f,\alpha)$ is \emph{\textbf{stable over $K$}} if $f^n - \alpha$ is irreducible over $K$ for all $n \geq 1$. We say that the pair $(f,\alpha)$ is \emph{\textbf{eventually stable over $K$}} if the number of irreducible factors of $f^n - \alpha$ over $K$ remains bounded as $n \to \infty$.
	\end{defin}
		Beyond their intrinsic interest, stability and eventual stability also play a key role in the Galois theory of iterated polynomials. For instance, recent results in \cite{BDGHT21, Bridy19} show that, under mild assumptions on $f$, eventual stability can, in certain cases, guarantee that the inverse limit of these Galois groups has finite index in an appropriate profinite overgroup. See below for further discussion of these Galois groups.\par

		Before stating our main results, we introduce a new family of polynomials defined over $K = \QQ(c_0)$ for a given Gleason parameter $c_0$. We will use these polynomials to describe the factorization of iterates of $f_{d,c_0}$ when $c_0$ is a Gleason parameter.
	\begin{defin}
		\label{def:F_{k,i}}
		Let $d$ be a prime. Let $c_0$ be a root of $G_{d,n}$ for some $n\geq 2$. Let $\zeta \neq 1$ be a $d$-th root of unity. Set $K = \QQ(c_0)$. For $k \geq 0$ and $1 \leq i \leq n - 1$, define the polynomial $F_{k,i}^{(d,c_0)}$ by
		\[
		F_{k,i}^{(d,c_0)} = \prod_{j=1}^{d-1} \left(f_{d,c_0}^k - \zeta^j a_i(c_0)\right).
		\]
	\end{defin}
	
	We are ready to state our main results.
	
	\begin{thm}
		\label{thm:preperiodic}
		Let $d$ be a prime. Let $\zeta\neq 1$ be a $d$-th root of unity, and let $c_0$ be a root of $G_{d,m,n}^{\zeta}$ for some $m \geq 2$, $n \geq 1$. Set $K = \QQ(c_0)$, and suppose that $\alpha \in K$ satisfies $v_{\mathfrak{p}}(\alpha) \geq 2$ for some prime $\mathfrak{p} \subset \mathcal{O}_K$ above $d$. Then the pair $(f_{d,c_0}, \alpha)$ is stable over $K$.
	\end{thm}
	
	\begin{remark}
		\label{remark:generalization}
		Theorem~\ref{thm:preperiodic} generalizes \cite[Corollary 1.2]{Gok19}, which shows stability of $f_{d,c_0}$ only in the case $n=1$ and $\alpha=0$.
	\end{remark}
	
	\begin{thm}
		\label{thm:periodic}
		Let $d$ be a prime. Let $c_0$ be a root of $G_{d,n}$ for some $n \geq 1$. Set $K = \QQ(c_0)$.
		\begin{enumerate}[label=(\alph*)]
			\item If $\alpha \in K$ satisfies $v_{\mathfrak{p}}(\alpha) = 1$ for some prime $\mathfrak{p} \subset \mathcal{O}_K$ above $d$, then the pair $(f_{d,c_0}, \alpha)$ is stable over $K$.
			\item Suppose further that $n\geq 2$. For any $k \geq 1$, write $k = nq + r$ for some $q \geq 0$ and $0 \leq r < n$. Then $f_{d,c_0}^k$ factors as
			\[
			f_{d,c_0}^k = \left(\prod_{j=0}^{q-1} \prod_{i=1}^{n-1} \left(F_{k - nj - i, n - i}^{(d,c_0)}\right)^{d^j}\right) \left((x - a_{n - r}(c_0)) \prod_{i=1}^r F_{r - i, n - i}^{(d,c_0)}\right)^{d^q}
			\]
			over $K$. Moreover, the polynomials of the form $F_{a,b}^{(d,c_0)}$ that appear in the above product are distinct and irreducible over $K$.
		\end{enumerate}
	\end{thm}
	
	It is a well-known fact that the pair \(( f_{d,c_0},0) \) is not eventually stable over \( K = \QQ(c_0) \) when \( c_0 \) is a Gleason parameter. Theorem~\ref{thm:periodic} provides a quantitative version of this fact:
	
	\begin{cor}
		\label{cor:quantitative}
		Let \( d \) be a prime, and let \( c_0 \) be a root of \( G_{d,n} \) for some \( n \geq 2\). Set \( K = \QQ(c_0) \). Then, for any \( k \geq 1 \), the polynomial \( f_{d,c_0}^k \) has \( k - \lfloor k/n \rfloor + 1 \) distinct irreducible factors over \( K \).
	\end{cor}
	
	\begin{proof}
		By direct computation using part (b) of Theorem~\ref{thm:periodic}, \( f_{d,c_0}^k \) has
		\[
		nq - q + r + 1 = k - q + 1
		\]
		distinct irreducible factors. Dividing both sides of the identity \( k = nq + r \) by \( n \), we obtain
		\[
		\frac{k}{n} = q + \frac{r}{n}.
		\]
		Since \( 0 \leq r \leq n - 1 \) by definition, this implies \( q = \lfloor k/n \rfloor \), which yields the desired result.
	\end{proof}
	
	\begin{remark}
		For $n\geq 2$, the expression \( k - \lfloor k/n \rfloor + 1 \) is clearly unbounded as \( k \to \infty \). If \( n = 1 \), then by the definition of \( G_{d,1} \), we have \( c_0 = 0 \), i.e., \( f_{d,0} = x^d \), so the pair $(f_{d,0},0)$ is clearly not eventually stable.
	\end{remark}
	
	Theorem~\ref{thm:preperiodic} and Theorem~\ref{thm:periodic} also have applications to \emph{arboreal Galois representations}, which we now describe. Let $K$ be a number field, and let $\alpha \in K$. Let $f \in K[x]$ be a polynomial of degree $d\geq 2$. One can construct an infinite rooted $d$-ary tree using iterates of the polynomial $f$ as follows. We place $\alpha$ at the root of the tree, and for any $n \geq 1$, we place the solutions of $f^n(x) = \alpha$ at the $n$-th level of the tree. Moreover, an edge is drawn from an element $\beta$ at the $n$-th level to an element $\theta$ at the $(n+1)$-th level if and only if $f(\theta) = \beta$. If one further assumes that $f^n - \alpha$ is separable for all $n \geq 1$, this construction gives a complete $d$-ary rooted tree. This tree is called a \emph{pre-image tree} of $f$, and is denoted by $T_{\infty}^d$. The absolute Galois group $\Gal(\overline{K}/K)$ has a natural action on $T_{\infty}^d$, which yields a homomorphism
	\[
	\rho: \Gal(\overline{K}/K) \rightarrow \Aut(T_{\infty}^d).
	\]
	
	The map $\rho$ is called an \emph{arboreal Galois representation}, and its image is denoted by $G_{\infty}(f,\alpha)$. Letting $G_n(f,\alpha)$ be the Galois group of $f^n - \alpha$ over $K$, this image can also be described as
	\[
	G_{\infty}(f,\alpha) = \varprojlim G_n(f,\alpha).
	\]
	The question of whether this image has finite index in $\Aut(T_{\infty}^d)$ is a major open problem in arithmetic dynamics. There are many parallels between arboreal Galois representations and the classical $\ell$-adic representations arising from elliptic curves, which has motivated much of the work in this area over the last two decades. See \cite[Section 5]{BIJMST19} and \cite{Jones14} for some general overviews of the subject.\par
	
	Based on existing results, it is generally expected that the group $G_{\infty}(f,\alpha)$ has finite index in $\Aut(T_{\infty}^d)$ unless $f$ belongs to some special family of polynomials. See \cite[Conjecture 3.11]{Jones14} for a precise conjecture in the quadratic case. In particular, it seems reasonable to expect that the group $G_{\infty}(f,\alpha)$ is very rarely abelian. In 2020, Andrews and Petsche \cite{AndrewsPetsche20} proposed the following conjectural description of all abelian arboreal Galois representations.
	
	\begin{conj}
		\label{conj:abelian}
		Let $K$ be a number field, $f$ a polynomial of degree $d\geq 2$, and $\alpha\in K$ be a non-exceptional point for $f$. Then $G_{\infty}(f,\alpha)$ is abelian if and only if there exists a root of unity $\zeta$ such that the pair $(f,\alpha)$ is $K^{\text{ab}}$-conjugated to either $(x^d,\zeta)$ or $(\pm T_d(x), \zeta+\zeta^{-1})$.
	\end{conj}
	Here, $K^{\text{ab}}$ is the maximal abelian extension of $K$, $T_d(x)$ is the Chebyshev polynomial of degree $d$, and $\alpha$ is called \emph{exceptional} if the full preimage set $\bigcup_{i=0}^{\infty} f^{-i}(\alpha)$ is finite.
	
	Recently, there has been substantial progress on Conjecture~\ref{conj:abelian}. Specifically, due to the works of Andrews-Petsche \cite{AndrewsPetsche20}, Ferraguti-Ostafe-Zannier \cite{FerOstZan24}, Ferraguti-Pagano \cite{FP23}, and Leung-Petsche \cite{LeungPetsche25}, the only remaining case is when $f$ is PCF and $\alpha$ is a preperiodic point of $f$.\par
	
	In the quadratic case, since any PCF quadratic polynomial is linearly conjugate to a polynomial of the form $f_{2,c_0}$ for a Misiurewicz or a Gleason parameter $c_0$, it remains to prove the conjecture for the pairs $(f_{2,c_0},\alpha)$, where $f_{2,c_0}$ is defined over a number field $K$, and $\alpha\in K$ is a preperiodic point of $f_{2,c_0}$.\par
	
	As a consequence of our work, we prove the following results concerning Conjecture~\ref{conj:abelian}.
	\begin{cor}
		\label{thm:arboreal-periodic}
		Let $d$ be a prime, and let $c_0$ be a root of $G_{d,n}$ for some positive integer $n$. Set $K = \QQ(c_0)$. Suppose that one of the following holds:
		\begin{itemize}
			\item $d = 2, n \geq 3$ and $v_{\mathfrak{p}}(\alpha)=1$ for some prime $\mathfrak{p} \subset \mathcal{O}_K$ lying above $d$;
			\item $d=2, \alpha=0$, and $n\geq 3$; 
			\item $d > 2$ and $v_{\mathfrak{p}}(\alpha) = 1$ for some prime $\mathfrak{p} \subset \mathcal{O}_K$ lying above $d$;
			\item $d > 2$, $\alpha = 0$, and $n \geq 2$.
		\end{itemize}
		Then $G_{\infty}(f_{d,c_0}, \alpha)$ is non-abelian.
	\end{cor}
	
	\begin{cor}
		\label{thm:arboreal-preperiodic}
		Let $d$ be a prime, and let $\zeta \neq 1$ be a $d$-th root of unity. Let $c_0$ be a root of $G_{d,m,n}^{\zeta}$ for some $m \geq 2$, $n \geq 1$. If $d = 2$, suppose further that $n \geq 3$. Set $K = \QQ(c_0)$, and suppose that $\alpha \in K$ satisfies $v_{\mathfrak{p}}(\alpha) \geq 2$ for some prime $\mathfrak{p} \subset \mathcal{O}_K$ lying above $d$. Then $G_{\infty}(f_{d,c_0}, \alpha)$ is non-abelian.
	\end{cor}
	
	In particular, Corollary~\ref{thm:arboreal-preperiodic} shows that $G_{\infty}(f_{d,c_0}, 0)$ is non-abelian when $d$ is an odd prime (resp. $2$) and $c_0$ is a Misiurewicz parameter of type $(m,n)$ for any $n \geq 1$ (resp. $n\geq 3$). Since $f_{d,c_0}$ is PCF and $0$ is a strictly preperiodic point for $f_{d,c_0}$, this proves new cases of Conjecture~\ref{conj:abelian}. On the other hand, Corollary~\ref{thm:arboreal-periodic} provides an alternative proof of an already known case (see \cite[Theorem 1.2]{FP23}).\par
	

	The outline of the paper is as follows: In Section~\ref{section:2}, we provide the necessary background and prove auxiliary lemmas that will be crucial for the proofs of the main results. In Section~\ref{section:3}, we prove Theorem~\ref{thm:preperiodic} and Theorem~\ref{thm:periodic}. Finally, in Section~\ref{section:4}, we prove Corollary~\ref{thm:arboreal-periodic} and Corollary~\ref{thm:arboreal-preperiodic}.
	
	\section{Background and auxiliary lemmas}
	\label{section:2}
	We start this section by recalling a result from \cite{Gok19}, which will play a significant role in the proofs of Theorem~\ref{thm:preperiodic} and Corollary~\ref{thm:arboreal-preperiodic}. Throughout the paper, for a number field $K$ and $a\in \mathcal{O}_K$ (the ring of integers of $K$), we denote by $\la a\ra$ the principal ideal generated by $a$.
	\begin{thm}[\cite{Gok19}]
		\label{thm:power-relation}
		Let $f_{d,c_0}=x^d+c_0\in \CC[x]$ be a PCF polynomial having exact type $(m,n)$ with $m\geq 2$, $n\geq 1$. Set $K=\QQ(c_0)$. Then the following holds.
		\begin{enumerate}[(a)]
			\item If $n\nmid i$, then $a_i(c_0)$ is an algebraic unit.
			\item If $d$ is a prime and $n\text{ }|\text{ }i$, then one has $\la a_i(c_0)\ra^A = \la d\ra$, where
			
			\[A = \begin{cases}
				d^{m-1}(d-1)
				& \text{ if } n|m-1,
				\\
				(d^{m-1}-1)(d-1) & \text{ if } n \nmid m-1.
			\end{cases}\]
			
		\end{enumerate}
	\end{thm}
	The following two lemmas describe the general form of the iterates of $f_{d,c_0}$ when $c_0$ is a Gleason or Misiurewicz parameter. In the proofs of Theorems~\ref{thm:preperiodic} and \ref{thm:periodic}, these lemmas will enable us to produce Eisenstein polynomials by building on results from \cite{Gok19,Gok20} and by using an appropriate iterate of $f_{d,c_0}$.
	\begin{lemma}
		\label{lem:expansion_preperiodic}
		Let $d$ be a prime, $k \geq 1$, and let $\zeta \neq 1$ be a $d$-th root of unity. Suppose $c_0$ is a root of $G_{d,m,n}^{\zeta}$ for some $m\geq 2$, $n\geq 1$. Set $K=\QQ(c_0)$. Then there exists a polynomial $F(x) \in \mathcal{O}_K[x]$ such that
		\[
		f_{d,c_0}^k(x) = x^{d^k} + d x^d F(x) + u a_i(c_0),
		\]
		where $i = \gcd(k,n)$ and $u\in \mathcal{O}_K$ is an algebraic unit.
	\end{lemma}
	
	\begin{proof}
		By direct expansion and using the fact that $\binom{d}{i}$ is divisible by $d$ for $i = 1, 2, \dots, d-1$, we have
		\[
		f_{d,c_0}^k(x) = x^{d^k} + d x^d F(x) + a_k(c_0)
		\]
		for some $F(x) \in \mathcal{O}_K[x]$. By Theorem~\ref{thm:power-relation}, the ideals generated by $a_k(c_0)$ and $a_i(c_0)$ coincide in $\mathcal{O}_K$, where $i = \gcd(k,n)$. The result follows immediately.
	\end{proof}
	\begin{lemma}
		\label{lem:expansion_periodic}
		Let $d$ be a prime and $k \geq 1$. Let $c_0$ be a root of $G_{d,n}$ for some $n\geq 1$. Set $K = \mathbb{Q}(c_0)$. Then there exists a polynomial $F(x) \in \mathcal{O}_K[x]$ such that
		\[
		f_{d,c_0}^k(x) = x^{d^k} + d x^d F(x) + a_i(c_0),
		\]
		where $i$ is the smallest positive integer satisfying $k \equiv i \pmod{n}$. In particular, if $k$ is divisible by $n$, then
		\[
		f_{d,c_0}^k(x) = x^{d^k} + d x^d F(x)
		\]
		for some $F(x) \in \mathcal{O}_K[x]$.
	\end{lemma}
	
	\begin{proof}
		Similar to the proof of Lemma~\ref{lem:expansion_preperiodic}, by direct expansion we have
		\[
		f_{d,c_0}^k(x) = x^{d^k} + d x^d F(x) + a_k(c_0)
		\]
		for some $F(x) \in \mathcal{O}_K[x]$. Since $c_0$ is a root of the Gleason polynomial $G_{d,n}$, the sequence $(a_j(c_0))_j$ is periodic with period $n$, and so
		\[
		a_{nq + r}(c_0) = a_r(c_0)
		\]
		for all integers $q \geq 0$ and $0 \leq r \leq n-1$. Thus, the first part of the statement follows by taking $i$ as the minimal positive residue of $k$ modulo $n$. In particular, when $k$ is divisible by $n$, we have $a_k(c_0) = 0$ by periodicity, which proves the second part.
	\end{proof}
	The next lemma shows that the extensions of \(\mathbb{Q}\) generated by the roots of $G_{d,n}$ never contain any irrational $d$-th roots of unity.
	
	\begin{lemma}
		\label{lem:trivial_cyclo}
		Let \(d\) be a prime and \(k \geq 1\). Let \(c_0\) be a root of \(G_{d,n}\) for some $n\geq 1$. Set \(K = \mathbb{Q}(c_0)\). Then, for any primitive \(d\)-th root of unity \(\zeta\), we have
		\[
		K \cap \mathbb{Q}(\zeta) = \mathbb{Q}.
		\]
	\end{lemma}
	
	\begin{proof}
		Note that \(d\) is totally ramified in the extension \(\mathbb{Q}(\zeta)/\mathbb{Q}\), hence \(d\) must be totally ramified in any nontrivial subextension as well. Let \(g \in \mathbb{Z}[x]\) be the minimal polynomial of \(c_0\). Since \(g\) divides \(G_{d,n}\), which is known to have only simple roots modulo \(d\) (see, for instance, \cite[Lemma 3]{Buff18}), it follows that \(g\) also has simple roots modulo \(d\). Hence, \(d \nmid \operatorname{disc}(g)\), implying \(d \nmid \operatorname{Disc}(K)\). Thus, \(d\) is unramified in \(K\), which implies that \(d\) is unramified in the intersection \(K \cap \mathbb{Q}(\zeta)\) as well. Since \(d\) is totally ramified in any nontrivial subextension of \(\mathbb{Q}(\zeta)/\mathbb{Q}\), this forces the intersection \(K \cap \mathbb{Q}(\zeta)\) to be trivial, as desired.
	\end{proof}
	
	We now establish two facts about the polynomials \(F_{k,i}^{(d,c)}\) defined in Definition~\ref{def:F_{k,i}}, which are essential for the proof of Theorem~\ref{thm:periodic}.
	
	\begin{lemma}
		\label{lemma:gcd-F_{k,i}}
		Let $d$ be a prime. Let \(c_0\) be a root of \(G_{d,n}\) for some $n\geq 2$. Set \(K = \mathbb{Q}(c_0)\). Then the following hold:
		\begin{enumerate}[label=(\alph*)]
			\item \(F_{k,i}^{(d,c_0)}\) is defined over \(K\) for any \(k \geq 0\), \(1\leq i \leq n-1\).
			\item Let \(k_1, k_2 \geq 0\) and \(1 \leq i_1, i_2 \leq n - 1\). Then \(\gcd(F_{k_1,i_1}^{(d,c_0)}, F_{k_2,i_2}^{(d,c_0)}) > 1\) if and only if \((k_1, i_1) = (k_2, i_2)\).
		\end{enumerate}
	\end{lemma}
	
	\begin{proof}
		
		\textbf{Part (a).} Let $S=\{\zeta, \zeta^2,\dots, \zeta^{d-1}\}$. Note that $\sigma(S)=S$ for any $\sigma\in \Gal(\overline{K}/K)$. Since $f_{d,c_0}^k$ is defined over $K$ and $a_i(c_0)\in K$, it follows that any $\sigma\in \Gal(\overline{K}/K)$ leaves $F_{k,i}^{(d,c_0)}$ fixed, which implies that $F_{k,i}^{(d,c_0)}$ is defined over $K$.
		
		\textbf{Part (b).} Define $G_{k,i,j}^{(d,c_0)}=f_{d,c_0}^k-\zeta^ja_i(c_0)\in K(\zeta)[x]$ for $1\leq j\leq d-1$. We will prove that $\gcd(G_{k_1,i_1,j_1}^{(d,c_0)}, G_{k_2,i_2,j_2}^{(d,c_0)})>1$ if and only if $(k_1,i_1,j_1)=(k_2,i_2,j_2)$, which will automatically imply the result. First, suppose that $\alpha\in \overline{K}$ is a common root for $G_{k_1,i_1,j_1}^{(d,c_0)}$ and $G_{k_2,i_2,j_2}^{(d,c_0)}$. Thus, we have
		\begin{equation}
			\label{eq:i_1,j_1,i_2,j_2,alpha}
			f_{d,c_0}^{k_1}(\alpha) = \zeta^{j_1}a_{i_1}(c_0),\text{ }f_{d,c_0}^{k_2}(\alpha) = \zeta^{j_2}a_{i_2}(c_0).
		\end{equation}
		Assume without loss of generality $k_1\geq k_2$. For any $N\geq 1$, iterating each equation by $f_{d,c_0}^N$, we obtain
		\begin{equation}
			\label{eq:iterate by f^N}
			f_{d,c_0}^{k_1+N}(\alpha) = a_{i_1+N}(c_0),\text{ }f_{d,c_0}^{k_2+N}(\alpha) = a_{i_2+N}(c_0).
		\end{equation}
		Using Eq.~\ref{eq:iterate by f^N}, we get
		\[f_{d,c_0}^{k_1+N}(\alpha) = f_{d,c_0}^{k_1-k_2}(f_{d,c_0}^{k_2+N}(\alpha))=f_{d,c_0}^{k_1-k_2}(a_{i_2+N}(c_0)) = a_{k_1-k_2+i_2+N}(c_0).\]
		Combining this with the first part of Eq.~\ref{eq:iterate by f^N}, then, we conclude
		\begin{equation}
			\label{eq:k_1-k_2=i_1-i_2}
			k_1-k_2\equiv i_1-i_2 (\text{mod }n).
		\end{equation}
		Now, if $k_1>k_2$, apply $f_{d,c_0}^{k_1-k_2}$ to both sides of the equation $f_{d,c_0}^{k_2}(\alpha)=\zeta^{j_2}a_{i_2}(c_0)$ to get
		\[f_{d,c_0}^{k_1}(\alpha) = a_{i_2+k_1-k_2}(c_0)= a_{i_1}(c_0),\]
		where we used Eq.~\ref{eq:k_1-k_2=i_1-i_2} in the last equality. This contradicts the equation $f_{d,c_0}^{k_1}(\alpha)=\zeta^{j_1}a_{i_1}(c_0)$ since $\zeta\neq 1$. We conclude $k_1=k_2$. But, using this in Eq.~\ref{eq:k_1-k_2=i_1-i_2}, we also have $i_1 = i_2$. Combining this with Eq.~\ref{eq:i_1,j_1,i_2,j_2,alpha} we also get $j_1=j_2$, finishing the proof of Lemma~\ref{lemma:gcd-F_{k,i}}. 
	\end{proof}
	The following is a simple fact from arithmetic dynamics. Since we have not found a published proof, we include one here for the convenience of the reader.
	
	\begin{lemma}
		\label{lemma:irreducible-infinite-iterate}
		Let $K$ be a field, and let $f\in K[x]$ be a non-constant polynomial. Let $\alpha\in K$. If $f^i-\alpha$ is irreducible over $K$ for infinitely many positive integers $i$, then $f^i-\alpha$ is irreducible over $K$ for all $i\geq 1$.
	\end{lemma}
	\begin{proof}
		Let $k$ be an arbitrary positive integer. Choose a positive integer $i\geq k$ such that $f^i-\alpha$ is irreducible over $K$. Since we have
		\[f^i-\alpha = (f^k-\alpha)\circ f^{i-k},\]
		this forces $f^k-\alpha$ to be irreducible over $K$ as well, as desired.
	\end{proof}
	We now quote Capelli's lemma, a standard result used to prove the irreducibility of compositions of polynomials. The corollary that follows is a key tool in proving that the polynomials \(F_{k,i}^{(d,c_0)}\) are irreducible over \(K := \mathbb{Q}(c_0)\) for any Gleason parameter \(c_0\). This fact will be used to describe the complete factorization of iterates of \(f_{d,c_0}\).
	\begin{lemma}[Capelli's Lemma]
		\label{lemma:Capelli}
		Let $K$ be a field, $f(x), g(x)\in K[x]$, and let $\beta\in \overline{K}$ be any root of $g(x)$. Then $g(f(x))$ is irreducible over $K$ if and only if both $g$ is irreducible over $K$ and $f(x)-\beta$ is irreducible over $K(\beta)$.
	\end{lemma}
	\begin{cor}
		\label{cor:capelli}
		Let \(K\) be a field, and let \(f, u \in K[x]\) be polynomials over \(K\). Let \(L\) be a Galois extension of \(K\). For some \(\alpha \in K\), let \(\mathcal{O}(\alpha) = \{\alpha_1, \alpha_2, \dots, \alpha_k\}\) be the \(\operatorname{Gal}(L/K)\)-orbit of \(\alpha\). Suppose that \(f - u(\alpha_i)\) is irreducible over \(K(\alpha_i)\) for some \(1 \leq i \leq k\). Then the polynomial
		\[
		h := \prod_{i=1}^k (f - u(\alpha_i)) \in K[x]
		\]
		is irreducible.
	\end{cor}
	
	\begin{proof}
		Note that \(h\) is defined over \(K\) because \(f\) and \(u\) are defined over \(K\), and any element of \(\operatorname{Gal}(L/K)\) permutes the values \(u(\alpha_i)\). Let \(g\) be the minimal polynomial of \(u(\alpha)\) over \(K\). Since the \(\operatorname{Gal}(L/K)\)-orbit of \(u(\alpha)\) is \(\{u(\alpha_1), u(\alpha_2), \dots, u(\alpha_k)\}\), we have \(h = g \circ f\). The result now immediately follows from Lemma~\ref{lemma:Capelli}.
	\end{proof}
	\section{Proofs of main results}
	\label{section:3}
	We start this section by giving a factorization of the iterates \(f_{d,c_0}^k\) when \(c_0\) is a Gleason parameter. For $n \geq 0$, $\alpha \in K$ and $f\in K[x]$, define the set
	\[
	R_{n,\alpha}(f) = \{\beta \in \overline{K} \mid f^n(\beta) = \alpha \}.
	\]
	\begin{prop}
		\label{prop:factorization-f_{d,c}^k}
		Let $d$ be a prime. Let \(\zeta \neq 1\) be a \(d\)-th root of unity. Let \(c_0\) be a root of \(G_{d,n}\) for some $n\geq 2$. Let \(k\) be a positive integer, and write \(k = nq + r\) for some integers \(q \geq 0\) and \(0 \leq r < n\). Then the following holds:
		\begin{equation}
			\label{eq:factorization of f_{d,c}^k}
			f_{d,c_0}^k = \left(\prod_{j=0}^{q-1} \prod_{i=1}^{n-1} \left(F_{k - nj - i,\, n - i}^{(d,c_0)}\right)^{d^j}\right)
			\left( (x - a_{n-r}(c_0)) \prod_{i=1}^r F_{r - i,\, n - i}^{(d,c_0)} \right)^{d^q}.
		\end{equation}
		Here, we understand an empty product to be \(1\).
	\end{prop}
	\begin{proof}
		Fix \(r\), and proceed by induction on \(q\). First, consider the base case \(q=0\). Then \(k = r\). Since \(0\) is periodic under \(f_{d,c_0}\) with period \(n\), we have
		\[
		f_{d,c_0}(a_{n-1}(c_0)) = f_{d,c_0}(\zeta a_{n-1}(c_0)) = \dots = f_{d,c_0}(\zeta^{d-1} a_{n-1}(c_0)) = 0.
		\]
		Since \(\deg(f_{d,c_0}) = d\), this implies
		\[
		R_{1,0}(f_{d,c_0}) = \{a_{n-1}(c_0), \zeta a_{n-1}(c_0), \dots, \zeta^{d-1} a_{n-1}(c_0)\}.
		\]
		Therefore, we have the factorization
		\begin{align}
			\begin{split}
				\label{eq:f_{d,c}^k}
				f_{d,c_0}^k = f_{d,c_0}^r = f_{d,c_0} \circ f_{d,c_0}^{r-1} 
				&= (f_{d,c_0}^{r-1} - a_{n-1}(c_0)) \prod_{i=1}^{d-1} (f_{d,c_0}^{r-1} - \zeta^i a_{n-1}(c_0)) \\
				&= (f_{d,c_0}^{r-1} - a_{n-1}(c_0)) F_{r-1, n-1}^{(d,c_0)}.
			\end{split}
		\end{align}
		
		Now consider any polynomial of the form \(f_{d,c_0}^i - a_j(c_0)\), where \(i \geq 1\), \(j \geq 2\). We have
		\[
		f_{d,c_0}(a_{j-1}(c_0)) = f_{d,c_0}(\zeta a_{j-1}(c_0)) = \dots = f_{d,c_0}(\zeta^{d-1} a_{j-1}(c_0)) = a_j(c_0),
		\]
		which gives
		\[
		R_{1,a_j(c_0)}(f_{d,c_0}) = \{a_{j-1}(c_0), \zeta a_{j-1}(c_0), \dots, \zeta^{d-1} a_{j-1}(c_0)\}.
		\]
		Therefore,
		\begin{align}
			\begin{split}
				\label{eq:f_{d,c}^k-a_j}
				f_{d,c_0}^i - a_j(c_0) = f_{d,c_0} \circ f_{d,c_0}^{i-1} - a_j(c_0) 
				&= (f_{d,c_0}^{i-1} - a_{j-1}(c_0)) \prod_{\ell=1}^{d-1} (f_{d,c_0}^{i-1} - \zeta^\ell a_{j-1}(c_0)) \\
				&= (f_{d,c_0}^{i-1} - a_{j-1}(c_0)) F_{i-1, j-1}^{(d,c_0)}.
			\end{split}
		\end{align}
		
		Also, if \(j=1\), then
		\begin{equation}
			\label{eq:iterate-j=1}
			f_{d,c_0}^i - a_j(c_0) = (f_{d,c_0}^{i-1})^d.
		\end{equation}
		
		Using Eq.~\eqref{eq:f_{d,c}^k-a_j} for \((i,j) = (r-1,n-2), (r-2,n-2), \dots, (1,n-r+1)\) together with Eq.~\eqref{eq:f_{d,c}^k} gives
		\begin{equation}
			\label{eq:base case}
			f_{d,c_0}^k = f_{d,c_0}^r = (f_{d,c_0}^0 - a_{n-r}(c_0)) \prod_{i=1}^r F_{r-i, n-i}^{(d,c_0)} = (x - a_{n-r}(c_0)) \prod_{i=1}^r F_{r-i, n-i}^{(d,c_0)},
		\end{equation}
		which proves the base case.
		
		Assume the statement holds for \(q = A \geq 0\), i.e.,
		\begin{equation}
			\label{eq:inductive-step}
			f_{d,c_0}^{nA + r} = \left(\prod_{j=0}^{A-1} \prod_{i=1}^{n-1} \left(F_{n(A-j) + r - i, n - i}^{(d,c_0)}\right)^{d^j}\right) \left((x - a_{n-r}(c_0)) \prod_{i=1}^r F_{r-i, n-i}^{(d,c_0)}\right)^{d^A}.
		\end{equation}
		
		To prove the statement for \(q = A + 1\), note that for any \(a,b \geq 0\) and \(1 \leq e \leq n-1\), we have
		\begin{equation}
			\label{eq:compose with f_{d,c}}
			F_{a,e}^{(d,c_0)} \circ f_{d,c_0}^b = \prod_{j=1}^{d-1} (f_{d,c_0}^{a+b} - \zeta^j a_e(c_0)) = F_{a+b, e}^{(d,c_0)}.
		\end{equation}
		
		Now,
		
		\begin{align*}
			f_{d,c_0}^{n(A+1) + r} &= f_{d,c_0}^{nA + r} \circ f_{d,c_0}^n \\
			&= \left(\prod_{j=0}^{A-1} \prod_{i=1}^{n-1} \left(F_{n(A-j) + r - i, n - i}^{(d,c_0)} \circ f_{d,c_0}^n\right)^{d^j}\right)
			\left((f_{d,c_0}^n - a_{n-r}(c_0)) \prod_{i=1}^r F_{r-i, n-i}^{(d,c_0)} \circ f_{d,c_0}^n\right)^{d^A} \\
			&= \left(\prod_{j=0}^{A-1} \prod_{i=1}^{n-1} \left(F_{n(A+1 - j) + r - i, n - i}^{(d,c_0)}\right)^{d^j}\right)
			\left((f_{d,c_0}^n - a_{n-r}(c_0)) \prod_{i=1}^r F_{n + r - i, n - i}^{(d,c_0)}\right)^{d^A} \\
			&= \left(\prod_{j=0}^{A-1} \prod_{i=1}^{n-1} \left(F_{n(A+1 - j) + r - i, n - i}^{(d,c_0)}\right)^{d^j}\right)
			\left((f_{d,c_0}^r)^d \left(\prod_{j=r+1}^{n-1} F_{n + r - j, n - j}^{(d,c_0)}\right) \left(\prod_{i=1}^r F_{n + r - i, n - i}^{(d,c_0)}\right)\right)^{d^A} \\
			&= \left(\prod_{j=0}^{A-1} \prod_{i=1}^{n-1} \left(F_{n(A+1 - j) + r - i, n - i}^{(d,c_0)}\right)^{d^j}\right)
			\left((f_{d,c_0}^r)^d \prod_{i=1}^{n-1} F_{n + r - i, n - i}^{(d,c_0)}\right)^{d^A} \\
			&= \left(\prod_{j=0}^A \prod_{i=1}^{n-1} \left(F_{n(A+1 - j) + r - i, n - i}^{(d,c_0)}\right)^{d^j}\right) (f_{d,c_0}^r)^{d^{A+1}} \\
			&= \left(\prod_{j=0}^A \prod_{i=1}^{n-1} \left(F_{n(A+1 - j) + r - i, n - i}^{(d,c_0)}\right)^{d^j}\right) \left((x - a_{n-r}(c_0)) \prod_{i=1}^r F_{r-i, n-i}^{(d,c_0)}\right)^{d^{A+1}},
		\end{align*}
		which completes the induction. Note that we used Eq.~\eqref{eq:compose with f_{d,c}} in the third equality, Eq.~\eqref{eq:f_{d,c}^k-a_j} (for $(i,j)= (n,n-r), (n-1, n-r-1), \dots, (r+2, 2)$) and Eq.~\eqref{eq:iterate-j=1} (for $(i,j) = (r+1,1)$) in the fourth equality, and Eq.~\eqref{eq:base case} in the last equality.
		
		This completes the proof of Proposition~\ref{prop:factorization-f_{d,c}^k}.
	\end{proof}
	We are ready to prove Theorem~\ref{thm:preperiodic} and Theorem~\ref{thm:periodic}.
	\begin{proof}[Proof of Theorem~\ref{thm:preperiodic}]
		We will show that for any integer \(i \geq 1\), the polynomial \(f_{d,c_0}^{ni} - \alpha\) is Eisenstein at some prime ideal \(\mathfrak{p} \subset \mathcal{O}_K\) lying above the prime \(d\). Since Eisenstein polynomials are irreducible, the irreducibility of $f^k-\alpha$ for all $k\geq 1$ will then follow immediately by Lemma~\ref{lemma:irreducible-infinite-iterate}.
		
		By Lemma~\ref{lem:expansion_preperiodic}, the constant coefficient of \(f_{d,c_0}^{ni} - \alpha\) is of the form \(u a_n(c_0) - \alpha\) for some algebraic unit \(u \in \mathcal{O}_K\). From Theorem~\ref{thm:power-relation}(b) and \cite[Corollary 3.5]{Gok20}, the element \(a_n(c_0)\) is square-free and non-unit in \(\mathcal{O}_K\). Furthermore, Theorem~\ref{thm:power-relation}(b) also implies that the principal ideal \(\la a_n(c_0)\ra\) satisfies
		\[
		\la ua_n(c_0)\ra^A = \la a_n(c_0)\ra^A = \la d\ra
		\]
		for some positive integer \(A\). Therefore, for any prime ideal \(\mathfrak{p} \subset \mathcal{O}_K\) dividing \(d\), we have
		\[
		v_{\mathfrak{p}}(u a_n(c_0)) = 1.
		\]
		By hypothesis, \(\alpha\) satisfies \(v_{\mathfrak{p}}(\alpha) \geq 2\) for some prime ideal \(\mathfrak{p} \mid d\). The non-archimedean property of valuations then implies
		\[
		v_{\mathfrak{p}}(u a_n(c_0) - \alpha) = \min\{ v_{\mathfrak{p}}(u a_n(c_0)), v_{\mathfrak{p}}(\alpha) \} = 1.
		\]
		
		Hence, for such \(\mathfrak{p}\), we can write
		\[
		f_{d,c_0}^{ni} - \alpha = x^{d^{ni}} + d x^d F(x) + \beta
		\]
		where \(F(x) \in \mathcal{O}_K[x]\) and \(\beta \in \mathcal{O}_K\) satisfies \(v_{\mathfrak{p}}(\beta) = 1\). This shows that \(f_{d,c_0}^{ni} - \alpha\) is Eisenstein at \(\mathfrak{p}\). Consequently, each polynomial \(f^k - \alpha\) is irreducible over \(K\) for all \(k \geq 1\), completing the proof.
	\end{proof}
	
	\begin{proof}[Proof of Theorem~\ref{thm:periodic}]
		\textbf{Part (a).} For any integer \(i \geq 1\), by part (b) of Lemma~\ref{lem:expansion_periodic}, we have the expression
		\[
		f_{d,c_0}^{ni} - \alpha = x^{d^{ni}} + d x^d F(x) - \alpha
		\]
		for some \(F(x) \in \mathcal{O}_K[x]\). Since by assumption there exists a prime ideal \(\mathfrak{p} \subset \mathcal{O}_K\) above \(d\) such that
		\[
		v_{\mathfrak{p}}(\alpha) = 1,
		\]
		it follows that \(f_{d,c_0}^{ni} - \alpha\) is Eisenstein at \(\mathfrak{p}\). Hence, \(f_{d,c_0}^{ni} - \alpha\) is irreducible over \(K\) for all \(i \geq 1\). By Lemma~\ref{lemma:irreducible-infinite-iterate}, this irreducibility for infinitely many iterates implies that \(f_{d,c_0}\) is stable over \(K\), which completes the proof of part (a).\par
		
		\textbf{Part (b).} By Proposition~\ref{prop:factorization-f_{d,c}^k}, to complete the proof, it suffices to show that the factors appearing in Eq.~\ref{eq:factorization of f_{d,c}^k} are all distinct and irreducible over \(K\).
		
		First, note that the equality 
		\[
		k - n j - i = k - n j' - i'
		\]
		implies 
		\[
		n(j - j') = i' - i.
		\]
		Since \(1 \leq i, i' \leq n-1\), this can only hold if \(i = i'\) and \(j = j'\). Furthermore, we have \(k - n j - i > r\) for all \(1 \leq i \leq n-1\) and \(0 \leq j \leq q-1\). Therefore, Lemma~\ref{lemma:gcd-F_{k,i}} implies that all factors appearing in Eq.~\ref{eq:factorization of f_{d,c}^k} are relatively prime, and hence distinct.
		
		Next, consider the polynomials
		\[
		G_{i,j,\ell}^{(d,c_0)} = f_{d,c_0}^i - \zeta^{\ell} a_j(c_0) \in K(\zeta)[x]
		\]
		for \(1 \leq j \leq n-1\) and \(1 \leq \ell \leq d-1\). Let \(m\) be the smallest non-negative integer such that
		\[
		m + i \equiv j \pmod{n}.
		\]
		Then
		\[
		G_{i,j,\ell}^{(d,c_0)} \circ f_{d,c_0}^m = f_{d,c_0}^{m+i} - \zeta^{\ell} a_j(c_0).
		\]
		
		By Lemma~\ref{lem:expansion_periodic}, the constant coefficient of \(f_{d,c_0}^{m+i} - \zeta^{\ell} a_j(c_0)\) is
		\[
		a_j(c_0) - \zeta^{\ell} a_j(c_0) = (1 - \zeta^{\ell}) a_j(c_0).
		\]
		From \cite[Lemma 3.1]{Gok19}, \(a_j(c_0)\) is an algebraic unit in \(\mathcal{O}_K\). Consider the unique prime ideal \(\mathfrak{p} := (1 - \zeta^{\ell})\) above \(d\) in \(\mathbb{Q}(\zeta)\) \cite[Lemma 1.4]{Washington96}, and set \(L := K(\zeta)\). Recall from the proof of Lemma~\ref{lem:trivial_cyclo} that \(d\) is unramified in \(K\), hence \(\mathfrak{p}\) remains unramified in \(L\). Let \(\mathfrak{q}\) be a prime ideal in \(\mathcal{O}_L\) lying above \(\mathfrak{p}\). Since \(\mathfrak{p}\) is unramified in \(L\), we have
		\[
		v_{\mathfrak{q}}((1 - \zeta^{\ell}) a_j(c_0)) = 1.
		\]
		Thus, we can write
		\[
		G_{i,j,\ell}^{(d,c_0)} \circ f_{d,c_0}^m = f_{d,c_0}^{m+i} - \zeta^{\ell} a_j(c_0) = x^{d^{m+i}} + d x^d F(x) + \beta
		\]
		for some \(F(x) \in \mathcal{O}_L[x]\) and \(\beta \in \mathcal{O}_L\) with \(v_{\mathfrak{q}}(\beta) = 1\). It follows that \(G_{i,j,\ell}^{(d,c_0)} \circ f_{d,c_0}^m\) is Eisenstein at \(\mathfrak{q}\), hence irreducible over \(L\). Consequently, \(G_{i,j,\ell}^{(d,c_0)} = f_{d,c_0}^i - \zeta^{\ell} a_j(c_0)\) is irreducible over \(L\). Finally, applying Corollary~\ref{cor:capelli} with \(K = \mathbb{Q}(c_0)\), \(L = K(\zeta)\), \(f = f_{d,c_0}^i\), \(u = a_j(c_0) x\), and \(\alpha = \zeta^{\ell}\), we conclude that
		\[
		F_{i,j}^{(d,c_0)} = \prod_{\ell=1}^{d-1} \big(f_{d,c_0}^i - \zeta^{\ell} a_j(c_0)\big)
		\]
		is irreducible over \(K\), completing the proof of Theorem~\ref{thm:periodic}.
	\end{proof}
	\section{An application to arboreal Galois representations}
	\label{section:4}
	The goal of this section is to prove Corollary~\ref{thm:arboreal-periodic} and Corollary~\ref{thm:arboreal-preperiodic}. We begin by introducing some notation and recalling standard facts about arboreal Galois representations.
	
	The automorphism group of the infinite rooted $d$-ary tree, \(\Aut(T_{\infty}^{d})\), can be described as the inverse limit
	\[
	\Aut(T_{\infty}^d) = \varprojlim \Aut(T_n^d),
	\]
	where \(T_n^d\) denotes the rooted $d$-ary tree of height \(n\). The group \(\Aut(T_n^d)\) is isomorphic to the \(n\)-fold wreath product $[C_d]^n$, where \(C_d\) is the cyclic group of order $d$. For a field \(K\), an element \(\alpha \in K\), and a polynomial \(f \in K[x]\) of degree $d\geq 2$, it is well-known that the Galois group \(G_n(f,\alpha)\) embeds into \(\Aut(T_n^d)\) for all \(n \geq 1\).
	
	For a field \(K\), we denote by \(K^{\times 2}\) the subgroup of squares in \(K^\times = K \setminus \{0\}\). Finally, for an algebraic field extension \(L/K\), we write \(\mathrm{Nm}_{L/K}(\beta)\) for the field norm of \(\beta \in L\) over \(K\).
	
	The following lemma gives a useful formula for the discriminant of iterated polynomials, which will be crucial in the proofs of both corollaries.
	\begin{lemma}[\cite{BD24}]
		\label{lemma:discriminant of iterates}
		Let $K$ be a field. Let $f(x)\in K[x]$ be a polynomial of degree $d\geq 2$ with lead coefficient $A\in K^{\times}$, and let $x_0\in K$. Then for every $k\geq 1$, we have
		
		\[\Delta (f^k-x_0) = (-1)^{d^k(d-1)/2}d^{d^k}A^{d^{2k-1}-1}(\Delta(f^{k-1}-x_0))^d\prod_{f'(c)=0}^{} (f^k(c)-x_0),\]
		where the product is over all finite critical points of $f$, repeated according to multiplicity.
	\end{lemma}
	Next, we prove a simple fact from group theory that will play an important role in the proofs throughout this section. We denote by $S_n$ the symmetric group of degree $n$, and by $A_n$ the alternating group of degree $n$.
	
	\begin{lemma}
		\label{lemma:odd-alternating}
		Let $n>1$ be an odd positive integer, and let $G$ be an abelian transitive subgroup of $S_n$. Then $G$ must be contained in $A_n$.
	\end{lemma}
	
	\begin{proof}
		By \cite[Lemma 2]{AndrewsPetsche20}, we have $|G| = n$. Consider the sign homomorphism $\sgn : S_n \rightarrow \{-1, 1\}$. Suppose, for the sake of contradiction, that $G$ is not contained in $A_n$. This implies that the restriction homomorphism $\sgn|_G : G \rightarrow \{-1, 1\}$ is surjective. But then, by the first isomorphism theorem, we have $2 \mid n$, which contradicts the hypothesis. Hence, $G$ must be contained in $A_n$, as desired.
	\end{proof}
	
	\begin{remark}
		\label{remark:square-disc}
		Let $K$ be a field, and let $f \in K[x]$ be a polynomial of odd degree $d\geq 2$. If $f$ is irreducible over $K$ and the Galois group $G$ of $f$ over $K$ is abelian, then it follows from Lemma~\ref{lemma:odd-alternating} that $G$ must be contained in $A_d$. By a well-known fact from Galois theory, this implies that $\Delta(f) \in K^{\times 2}$. We will use this fact several times in the proofs of this section.
	\end{remark}
	
	We are finally ready to prove Corollary~\ref{thm:arboreal-periodic} and Corollary~\ref{thm:arboreal-preperiodic}.
	
	\begin{proof}[Proof of Corollary~\ref{thm:arboreal-periodic}]
		For the sake of contradiction, suppose that $G_{\infty}(f_{d,c_0}, \alpha)$ is abelian. We will prove each bullet point in the statement separately.
		
		\medskip
		\noindent \textbf{Case 1.} Suppose that $d = 2$, $n\geq 3$ and $v_{\mathfrak{p}}(\alpha) = 1$ for some prime $\mathfrak{p} \subset \mathcal{O}_K$ lying above $d$.
		
		\medskip
		Let $\beta$ be a root of $f_{2,c_0}^{n-2} - \alpha$, and let $G$ be the Galois group of $f_{2,c_0}^2 - \beta$ over $K(\beta)$. Since
		\[
		(f_{2,c_0}^{n-2} - \alpha) \circ f_{2,c_0}^2 = f_{2,c_0}^{n-2} \circ f_{2,c_0}^2 - \alpha = f_{2,c_0}^n - \alpha
		\]
		is irreducible over $K$ by Theorem~\ref{thm:periodic}, Lemma~\ref{lemma:Capelli} implies that $f_{2,c_0}^2 - \beta$ is irreducible over $K(\beta)$. Since $G$ is isomorphic to a subquotient of $G_n(f_{2,c_0}, \alpha)$, which is abelian, we conclude that $G$ is also abelian. Note that
		\[
		f_{2,c_0}^2 - \beta = x^4 + 2 f_{2,c_0}(0) x^2 + f_{2,c_0}^2(0) - \beta.
		\]
		By direct computation using \cite[Corollary 4.5]{Conrad}, it follows that either
		\begin{equation}
			\label{eq:disc-belongsK2}
			f_{2,c_0}^2(0) - \beta \in K(\beta)^{\times 2}
		\end{equation}
		or
		\begin{equation}
			\label{eq:belongsK2}
			4 \beta \big(f_{2,c_0}^2(0) - \beta\big) \in K(\beta)^{\times 2}
		\end{equation}
		must hold. By the multiplicativity of the norm, we have
		\[
		\mathrm{Nm}_{K(\beta)/K} \big( \beta \big(f_{2,c_0}^2(0) - \beta\big) \big) = \mathrm{Nm}_{K(\beta)/K}(\beta) \cdot \mathrm{Nm}_{K(\beta)/K}(f_{2,c_0}^2(0) - \beta).
		\]
		Since $\beta$ is a root of $f_{2,c_0}^{n-2} - \alpha$, and using the irreducibility of $f_{d,c_0}^{n-2}-\alpha$ over $K$,
		\[
		\mathrm{Nm}_{K(\beta)/K}(\beta) = f_{2,c_0}^{n-2}(0) - \alpha, \quad \mathrm{Nm}_{K(\beta)/K}(f_{2,c_0}^2(0) - \beta) = f_{2,c_0}^n(0) - \alpha.
		\]
		
		Using Eq.~\ref{eq:disc-belongsK2} and Eq.~\ref{eq:belongsK2}, it follows that either
		\[
		f_{2,c_0}^n(0) - \alpha \in K^{\times 2}
		\]
		or
		\[
		\big(f_{2,c_0}^{n-2}(0) - \alpha\big)\big(f_{2,c_0}^n(0) - \alpha\big) \in K^{\times 2}
		\]
		must hold. Since $f_{2,c_0}^n(0) = 0$ by periodicity, and $f_{2,c_0}^{n-2}(0)$ is an algebraic unit by \cite[Lemma 3.1]{Gok19} (recall the assumption $n\geq 3$), and given that $v_{\mathfrak{p}}(\alpha) = 1$, we have
		\[
		v_{\mathfrak{p}}(f_{2,c_0}^n(0) - \alpha) = v_{\mathfrak{p}}(\alpha) = 1
		\]
		and
		\[
		v_{\mathfrak{p}} \Big( \big(f_{2,c_0}^{n-2}(0) - \alpha\big)\big(f_{2,c_0}^n(0) - \alpha\big) \Big) = v_{\mathfrak{p}}(f_{2,c_0}^{n-2}(0) - \alpha) + v_{\mathfrak{p}}(f_{2,c_0}^n(0) - \alpha) = v_{\mathfrak{p}}(\alpha) = 1.
		\]
		
		Thus, neither of these two elements can be a square in $K$, which is a contradiction. Hence, $G_{\infty}(f_{2,c_0}, \alpha)$ cannot be abelian, completing the proof of this case.
		
		
		\medskip

		\textbf{Case 2.} Suppose  that $d=2$, $n\geq 3$ and $\alpha=0$.\\
		
		Note that 
		\[
		f_{2,c_0}^{n-2}(-a_2(c_0)) = f_{2,c_0}^{n-2}(a_2(c_0)) = f_{2,c_0}^n(0) = 0.
		\]
		Thus, $-a_2(c_0)$ lies in the preimage set $f_{2,c_0}^{-(n-2)}(0)$. Moreover, by the proof of Theorem~\ref{thm:periodic}(b), the polynomial
		\[
		f_{2,c_0}^i - (-a_2(c_0)) = f_{2,c_0}^i + a_2(c_0)
		\]
		is irreducible over $K$ for any $i \geq 1$. Let $G$ be the Galois group of $f_{2,c_0}^2 + a_2(c_0)$ over $K$. Since $G_{\infty}(f_{2,c_0}, 0)$ is assumed abelian, $G$ must also be abelian. We have
		\[
		f_{2,c_0}^2 + a_2(c_0) = x^4 + 2 a_1(c_0) x^2 + 2 a_2(c_0).
		\]
		Since $a_2(c_0)$ is an algebraic unit by \cite[Lemma 3.1]{Gok19} (and by recalling $n\geq 3$), and $2$ is unramified in $K$ (by the proof of Lemma~\ref{lem:trivial_cyclo}), it follows that 
		\[
		2 a_2(c_0) \notin K^{\times 2}.
		\]
		Therefore, by \cite[Corollary 4.5]{Conrad}, the Galois group of $f_{2,c_0}^2 + a_2(c_0)$ over $K$ must be isomorphic to $\ZZ/4\ZZ$. By the same corollary, we conclude
		\[
		\bigl( (2 a_1(c_0))^2 - 8 a_2(c_0) \bigr)(2 a_2(c_0)) = 4(a_1(c_0)^2 - 2 a_2(c_0)) \cdot 2 a_2(c_0) \in K^{\times 2}.
		\]
		Note that $a_1(c_0)$ is an algebraic unit by \cite[Lemma 3.1]{Gok19}, so
		\[
		v_{\mathfrak{p}}(a_1(c_0)^2 - 2 a_2(c_0)) = 0.
		\]
		Since $2$ is unramified in $K$ and $a_2(c_0)$ is an algebraic unit, we also have
		\[
		v_{\mathfrak{p}}(4) = 2, \quad v_{\mathfrak{p}}(2 a_2(c_0)) = 1.
		\]
		Thus,
		\[
		v_{\mathfrak{p}} \bigl( 4(a_1(c_0)^2 - 2 a_2(c_0)) \cdot 2 a_2(c_0) \bigr) = 2 + 0 + 1 = 3.
		\]
		Therefore,
		\[
		4(a_1(c_0)^2 - 2 a_2(c_0)) \cdot 2 a_2(c_0) \notin K^{\times 2},
		\]
		a contradiction. Hence, $G_{\infty}(f_{2,c_0}, 0)$ is non-abelian, as desired.
		
		\medskip
		\textbf{Case 3.} Suppose that $d>2$ and $v_{\mathfrak{p}}(\alpha)=1$ for some prime $\mathfrak{p}\subset \mathcal{O}_K$ lying above $d$. 
		
		\medskip
		Note that $f_{d,c_0}^i - \alpha$ is irreducible for any $i \geq 1$ by Theorem~\ref{thm:periodic}. Since $G_i(f_{d,c_0},\alpha)$ is also abelian for any $i \geq 1$ by assumption, Remark~\ref{remark:square-disc} yields
		\[
		\Delta(f_{d,c_0}^i - \alpha) \in K^{\times 2}
		\]
		for any $i \geq 1$. Now, take any $j >1$ such that $n \nmid j$. Using Lemma~\ref{lemma:discriminant of iterates} with $k = j - 1$ and $k = j$, we conclude
		\begin{equation}
			\label{eq:square-ratio-disc}
			\frac{\Delta(f_{d,c_0}^j - \alpha)}{(\Delta(f_{d,c_0}^{j-1} - \alpha))^d} = \pm d^{d^j}(f_{d,c_0}^j(0) - \alpha) \in K^{\times 2}.
		\end{equation}
		Recall from the proof of Lemma~\ref{lem:trivial_cyclo} that $d$ is unramified in $K$, i.e., $v_{\mathfrak{p}}(d) = 1$. Moreover, since $f_{d,c_0}^j(0)$ is an algebraic unit by \cite[Lemma 3.1]{Gok19} and $v_{\mathfrak{p}}(\alpha) = 1$ by assumption, we have $v_{\mathfrak{p}}(f_{d,c_0}^j(0) - \alpha) = 0$. Hence,
		\[
		v_{\mathfrak{p}}(d^{d^j}(f_{d,c_0}^j(0) - \alpha)) = d^j v_{\mathfrak{p}}(d) + v_{\mathfrak{p}}(f_{d,c_0}^j(0) - \alpha) = d^j,
		\]
		which is odd. Thus, $\pm d^{d^j}(f_{d,c_0}^j(0) - \alpha) \notin K^{\times 2}$, contradicting Eq.~\ref{eq:square-ratio-disc}. Hence, $G_{\infty}(f_{d,c_0},\alpha)$ cannot be abelian, completing the proof of this case.
		
		\medskip
		\textbf{Case 4.} Suppose that $d>2$, $n\geq 2$ and $\alpha = 0$. 
		
		\medskip
		Let $\zeta \ne 1$ be a $d$-th root of unity. Note that
		\[
		f_{d,c_0}(\zeta a_{n-1}(c_0)) = f_{d,c_0}(a_{n-1}(c_0)) = 0.
		\]
		Therefore, $\zeta a_{n-1}(c_0)$ lies in the preimage set $f_{d,c_0}^{-1}(0)$. Furthermore, by the proof of Theorem~\ref{thm:periodic}, $f_{d,c_0}^i - \zeta a_{n-1}(c_0)$ is irreducible over $L := K(\zeta)$ for any $i \geq 1$. Let $G_i$ be the Galois group of $f_{d,c_0}^i - \zeta a_{n-1}(c_0)$ over $L$. Since $G_i$ must be abelian by assumption, Remark~\ref{remark:square-disc} yields
		\begin{equation}
			\label{eq:square-disc-alpha=0 case}
			\Delta(f_{d,c_0}^i - \zeta a_{n-1}(c_0)) \in L^{\times 2}
		\end{equation}
		for any $i \geq 1$. Now, applying Lemma~\ref{lemma:discriminant of iterates} for $f = f_{d,c_0}$, $k = 2n - 1$, and $x_0 = \zeta a_{n-1}(c_0)$, we obtain
		\[
		\Delta(f_{d,c_0}^{2n-1} - \zeta a_{n-1}(c_0)) = (-1)^{d^{2n-1}(d-1)/2} d^{d^{2n-1}} (\Delta(f_{d,c_0}^{2n-2} - \zeta a_{n-1}(c_0)))^d (f_{d,c_0}^{2n-1}(0) - \zeta a_{n-1}(c_0)).
		\]
		Recalling $f_{d,c_0}^{2n-1}(0) = a_{2n-1}(c_0) = a_{n-1}(c_0)$ and applying Eq.~\ref{eq:square-disc-alpha=0 case} for $i = 2n - 2$ and $i = 2n - 1$, we deduce
		\[
		\pm d^{d^{2n-1}}(1 - \zeta) a_{n-1}(c_0) \in L^{\times 2}.
		\]
		Since $\langle d \rangle = \langle 1 - \zeta \rangle^{d - 1}$ as principal ideals in $\QQ(\zeta)$ and $d$ is odd, we have $d \in L^{\times 2}$, which forces
		\[
		\pm (1 - \zeta) a_{n-1}(c_0) \in L^{\times 2}.
		\]
		However, since $d$ is unramified in $K$, the prime ideal $\langle 1 - \zeta \rangle$ of $\QQ(\zeta)$ is unramified in the extension $L / \QQ(\zeta)$. As $a_{n-1}(c_0)$ is an algebraic unit by \cite[Lemma 3.1]{Gok19} (since $n\geq 2$ by assumption), we conclude
		\[
		\pm (1 - \zeta) a_{n-1}(c_0) \notin L^{\times 2},
		\]
		which is the desired contradiction. Hence, $G_{\infty}(f_{d,c_0},0)$ is non-abelian, completing the proof of Corollary~\ref{thm:arboreal-periodic}.
	\end{proof}
	\begin{proof}[Proof of Corollary~\ref{thm:arboreal-preperiodic}]
		
		Suppose, for the sake of contradiction, that $G_{\infty}(f_{d,c_0},\alpha)$ is abelian. We will consider the cases $d=2$ and $d>2$ separately.
		
		\medskip
		\textbf{Case 1.} Suppose $d=2$ and $n\geq 3$. 
		
		\medskip
		Let $\beta$ be a root of $f_{2,c_0}^{n-2} - \alpha$. Similarly to the proof of Corollary~\ref{thm:arboreal-periodic}, since $f_{2,c_0}^n-\alpha$ is irreducible over $K$ by Theorem~\ref{thm:preperiodic}, Lemma~\ref{cor:capelli} implies that $f_{2,c_0}^2-\beta$ is irreducible over $K(\beta)$. Let $G$ be the Galois group of $f_{2,c_0}^2-\beta$ over $K(\beta)$. Since $G_{\infty}(f_{2,c_0},\alpha)$ is abelian by assumption, $G$ must also be abelian. We have
		\[
		f_{2,c_0}^2 - \beta = x^4 + 2 f_{2,c_0}(0) x^2 + f_{2,c_0}^2(0) - \beta.
		\]
		By direct computation using \cite[Corollary 4.5]{Conrad}, either
		\begin{equation}
			\label{eq:norm-square-disc}
			f_{2,c_0}^2(0)-\beta \in K(\beta)^{\times 2}
		\end{equation}
		or
		\begin{equation}
			\label{eq:norm-square}
			4 \beta \bigl(f_{2,c_0}^2(0) - \beta \bigr) \in K(\beta)^{\times 2}
		\end{equation}
		must hold. Note that
		\[
		\mathrm{Nm}_{K(\beta)/K}(f_{2,c_0}^2(0) - \beta) = f_{2,c_0}^n(0) - \alpha
		\]
		and
		\begin{align*}
			\mathrm{Nm}_{K(\beta)/K}\bigl(\beta(f_{2,c_0}^2(0) - \beta)\bigr) &= \mathrm{Nm}_{K(\beta)/K}(\beta) \cdot \mathrm{Nm}_{K(\beta)/K}(f_{2,c_0}^2(0) - \beta) \\
			&= (f_{2,c_0}^{n-2}(0) - \alpha)(f_{2,c_0}^n(0) - \alpha).
		\end{align*}
		Using Eq.~\ref{eq:norm-square-disc} and Eq.~\ref{eq:norm-square}, it follows that either
		\[
		f_{2,c_0}^n(0) - \alpha \in K^{\times 2}
		\]
		or
		\[
		(f_{2,c_0}^{n-2}(0) - \alpha)(f_{2,c_0}^n(0) - \alpha) \in K^{\times 2}
		\]
		must hold. By Theorem~\ref{thm:power-relation}(b), we have 
		\[
		v_{\mathfrak{p}}(f_{2,c_0}^n(0)) = v_{\mathfrak{p}}(a_n(c_0)) > 0,
		\]
		and by \cite[Corollary 3.5]{Gok20}, $a_n(c_0)$ is square-free in $\mathcal{O}_K$. Thus,
		\[
		v_{\mathfrak{p}}(f_{2,c_0}^n(0)) = 1.
		\]
		Since $v_{\mathfrak{p}}(\alpha) \geq 2$, the non-archimedean property of valuations implies
		\[
		v_{\mathfrak{p}}(f_{2,c_0}^n(0) - \alpha) = 1.
		\]
		This shows that
		\[
		f_{2,c_0}^n(0) - \alpha \notin K^{\times 2},
		\]
		which forces
		\[
		(f_{2,c_0}^{n-2}(0) - \alpha)(f_{2,c_0}^n(0) - \alpha) \in K^{\times 2}.
		\]
		On the other hand, by Theorem~\ref{thm:power-relation}(a), $f_{2,c_0}^{n-2}(0)$ is an algebraic unit (recall $n \geq 3$). Since $v_{\mathfrak{p}}(\alpha) = 2 > 0$, it follows that
		\[
		v_{\mathfrak{p}}(f_{2,c_0}^{n-2}(0) - \alpha) = 0.
		\]
		Hence,
		\[
		v_{\mathfrak{p}}\bigl( (f_{2,c_0}^{n-2}(0) - \alpha)(f_{2,c_0}^n(0) - \alpha) \bigr) = 1,
		\]
		implying
		\[
		(f_{2,c_0}^{n-2}(0) - \alpha)(f_{2,c_0}^n(0) - \alpha) \notin K^{\times 2},
		\]
		which is a contradiction.
		
		Therefore, $G_{\infty}(f_{2,c_0}, \alpha)$ cannot be abelian, completing the proof of the first case.\\
		
		\textbf{Case 2.} Suppose $d > 2$. 
		
		\medskip
		By Lemma~\ref{lemma:discriminant of iterates}, we have
		\begin{equation}
			\label{eq:discformula-preperiodic-case2}
			\Delta (f_{d,c_0}^{3n} - \alpha) = (-1)^{d^{3n}(d-1)/2} d^{d^{3n}} \bigl(\Delta(f_{d,c_0}^{3n-1} - \alpha)\bigr)^d (f_{d,c_0}^{3n}(0) - \alpha).
		\end{equation}
		Note that, using Lemma~\ref{lemma:discriminant of iterates} for $k = 1, 2, \dots, 3n-1$, since $d$ is odd, we have
		\begin{equation}
			\label{eq:discfor3n-1}
			\Delta(f_{d,c_0}^{3n-1} - \alpha) = \pm d^B C^2\prod_{i=1}^{3n-1} (f_{d,c_0}^i(0) - \alpha)
		\end{equation}
		for some positive integer $B$ and $C\in K^{\times}$. By Theorem~\ref{thm:power-relation}(b), since $d$ is odd, $v_{\mathfrak{p}}(d)$ is even. On the other hand, by Theorem~\ref{thm:power-relation}(a), $f_{d,c_0}^i(0)$ is an algebraic unit for any $i\in S:=\{1,2,\dots, 3n-1\}\setminus \{n,2n\}$. Since $v_{\mathfrak{p}}(\alpha) > 0$ by assumption, we obtain
		\begin{equation}
			\label{eq:valuation0-lastproof}
			v_{\mathfrak{p}}(f_{d,c_0}^i(0) - \alpha) = 0
		\end{equation}
		for all $i\in S$. Now, by Theorem~\ref{thm:power-relation}(b) and \cite[Corollary 3.5]{Gok20}, we have
		\[
		v_{\mathfrak{p}}(f_{d,c_0}^{jn}(0)) = v_{\mathfrak{p}}(a_{jn}(c_0)) = 1
		\]
		for any $j\geq 1$. Since $v_{\mathfrak{p}}(\alpha) = 2$ by hypothesis, we conclude that
		\begin{equation}
			\label{eq:valuation1 for n and 2n}
			v_{\mathfrak{p}}(f_{d,c_0}^{jn}(0) - \alpha) =  1
		\end{equation}
		for any $j\geq 1$. Using this for $j=1,2$ together with Eq.~\ref{eq:discfor3n-1} and Eq.~\ref{eq:valuation0-lastproof}, and recalling that $v_{\mathfrak{p}}(d)$ is even, we conclude that
		\[v_{\mathfrak{p}} \bigl(\Delta(f_{d,c_0}^{3n-1} - \alpha)\bigr) \text{ is even.}\]
		Using this together with Eq.~\ref{eq:discformula-preperiodic-case2} and Eq.~\ref{eq:valuation1 for n and 2n}, and again recalling that $v_{\mathfrak{p}}(d)$ is even, we conclude that
		\[v_{\mathfrak{p}} \bigl(\Delta(f_{d,c_0}^{3n} - \alpha)\bigr) \text{ is odd},\]
		which forces
		\begin{equation}
			\label{eq:nonsquare-disc-preperiodic-case2}
			\Delta(f_{d,c_0}^{3n} - \alpha) \notin K^{\times 2}.
		\end{equation}
		Finally, recall that $f_{d,c_0}^{3n} - \alpha$ is irreducible over $K$ by Theorem~\ref{thm:preperiodic}. Since $G_{3n}(f_{d,c_0}, \alpha)$ is also abelian by assumption, Eq.~\ref{eq:nonsquare-disc-preperiodic-case2} contradicts Remark~\ref{remark:square-disc}. Thus, $G_{\infty}(f_{d,c_0}, \alpha)$ cannot be abelian, finishing the proof of Corollary~\ref{thm:arboreal-preperiodic}.

	\end{proof}

\end{document}